\documentclass[12pt]{amsart}
\usepackage{graphicx}
\usepackage{amsfonts,amssymb,amscd,amsmath,enumerate,verbatim,calc}
\usepackage{enumitem}
\usepackage{mathtools}

\newtheorem{Theorem}{Theorem}[section]
\newtheorem{Lemma}[Theorem]{Lemma}

\newtheorem{Corollary}[Theorem]{Corollary}

\newtheorem{Remark}[Theorem]{Remark}

\begin{document}
\title{Automorphism related parameters of graph associated to a finite vector space}
\author{Hira Benish, Imran Javaid$^*$, M. Murtaza}
%\subjclass{Primary: , Secondary: }
\keywords{fixing number, fixing neighborhood, fixed number. 
\indent $^*$ Corresponding author: imran.javaid@bzu.edu.pk}
%\indent 2012 {\it Mathematics Subject Classification.} \\
\address{Centre for advanced studies in Pure and Applied Mathematics,
Bahauddin Zakariya University Multan, Pakistan\newline Email:
hira\_benish@yahoo.com, imran.javaid@bzu.edu.pk,
mahru830@gmail.com.}

%azeemhaider@bzu.edu.pk
%{imranjavaid45@gmail.com}{\rm }

\date{}
\maketitle
\begin{abstract}
In this paper, we discuss automorphism related parameters of a graph
associated to a finite vector space. The fixing neighborhood of a
pair $(u,v)$ of vertices of a graph $G$ is the set of all those
vertices $w$ of $G$, such that the orbits of $u$ and $v$ under the
action of stabilizer of $w$ are not equal. The fixed number of a
graph is the minimum number $k$ such that every subset of vertices
of $G$ of cardinality $k$ is a fixing set of $G$. We study some
properties of automorphisms of a graph associated to finite vector
space and find the fixing neighborhood of pair of vertices of the
graph. We also find the fixed number of the graph. It is shown that,
for every positive integer $N$, there exists a graph $G$ with
$fxd(G)-fix(G)\geq N$, where $fxd(G)$ is the fixed number and
$fix(G)$ is the fixing number of $G$.
\end{abstract}

\section{Preliminaries}
The notion of fixing set of graph has its origin in the idea of
symmetry breaking which was introduced by Albertson and Collins
\cite{Alb}. Erwin and Harary \cite{erw} introduced the fixing number of a graph
$G$. Fixing sets have been studied extensively to
destroy the automorphisms of various graphs
\cite{bou,bou1,cac,gib,haray}.

Fixing number and metric dimension are two closely related
invariants. The two invariants coincides on many families of graph
like path graph, cycle graph etc. The difference of two invariant on
families of graphs is studied in \cite{NBiggs,cac,erw} in terms of
the order of graph. Boutin \cite{bou} studied fixing sets in
connection with distance determining set (she used the name for
resolving sets). Arumugam \emph{et al.} \cite{AM} defined resolving
neighborhood to study the fractional metric dimension of graph. We
defined the fixing neighborhood of a pair of vertices of a graph in
\cite{HB} in order to study all those vertices of the graph whose
fixing can destroy an automorphism that maps the two vertices of the
pair on each other. Fixing neighborhood of a pair of vertices
contains all such vertices of graph that destroy all automorphisms
between those two vertices of graph. Motivated by the definition of
resolving number of a graph, Javaid {\it et al.} \cite{ijav} defined the fixed
number of a graph. The authors found
characterization and realizable results based on fixed number.

A new area of research in graph theory is associating graphs with
various algebraic structures. Beck \cite{Beck} initiated the study
of zero divisor graph of a commutative ring with unity to address
coloring problem. Bondy \emph{et al.} \cite{bates1,bates2} studied
commuting graphs associated to groups. Power graphs for groups and
semigroups were discussed in
\cite{Cameron,Chakrabarty,Moghaddamfar}. In \cite{Rad,Talebi}
intersection graphs were associated to vector spaces. In \cite{Das},
Das assigned a non-zero component graph to finite dimensional vector
spaces. The author studied its domination number and independence
number. In \cite{Das1}, the author discussed edge-connectivity and
chromatic number of the graph. Metric dimension and partition
dimension of non-zero component graph are studied in \cite{us}.
Fazil studied its fixing number in \cite{fa}. Murtaza {\it et al.}
studied locating-dominating sets and identifying codes of non-zero
component graph \cite{Murtaza}. In this paper we study some
properties of automorphisms of non zero component graph as defined
by Das.  We find the fixing neighborhood of pairs of vertices of the
graph and the fixed number of the graph.

Now we define some graph related terminology which is used in the
article: Let $G$ be a graph with the vertex set $V(G)$ and the edge
set $E(G)$. Two vertices $u$ and $v$ are \emph{adjacent}, if they
share an edge, otherwise they are called \emph{non-adjacent}. The
number of adjacent vertices of $v$ is called the degree of $v$ in
$G$. For a graph $G$, an {\it automorphism} of $G$ is a bijective
mapping $f$ on $V(G)$ such that $f(u)f(v)\in E(G)$ if and only if
$uv\in E(G)$. The set of all automorphisms of $G$ forms a group,
denoted by $\Gamma(G)$, under the operation of composition. For a
vertex $v$ of $G$, the set $\{f(v):f\in \Gamma(G)\}$ is the {\it
orbit} of $v$, denoted by $\mathcal{O}(v)$. If two vertices $u,v$
are mapped on each other under the action of an automorphism $g\in
\Gamma(G)$, then we write it $u\sim^gv$. If $u,v$ cannot be mapped on
each other, then we write $u\nsim^gv$. An automorphism $g\in
\Gamma(G)$ is said to \emph{fix} a vertex $v\in V(G)$ if $v\sim^g
v$. The {\it stabilizer} of a vertex $v$ is the set of all
automorphisms that fix $v$ and it is denoted by $\Gamma_v(G)$. Also,
$\Gamma_v(G)$ is a subgroup of $\Gamma(G)$. Let us consider sets
$S(G)=\{v\in V(G):$ $|\mathcal{O}(v)|\ge 2 \}$ and
$V_s(G)=\{(u,v)\in S(G)\times S(G):$ $u\neq v$ and
$\mathcal{O}(u)=\mathcal{O}(v)\}$. If $G$ is a rigid graph (i.e., a
graph with $\Gamma (G)={id}$), then $V_s(G)=\emptyset$. For $v\in
V(G)$, the subgroup $\Gamma_v(G)$ has a natural action on $V(G)$ and
the orbit of $u$ under this action is denoted by $\mathcal{O}_v(u)$
i.e., $\mathcal{O}_v(u)=\{g(u):g\in\Gamma_v(G)\}$. An automorphism
$g\in \Gamma(G)$ is said to {\it fix} a set $D\subseteq V(G)$ if for
all $v\in D$, $v\sim^g v$. The set of automorphisms that fix $D$,
denoted by $\Gamma_D(G)$, is a subgroup of $\Gamma(G)$ and
$\Gamma_D(G)=\cap_{v\in D}\Gamma_v(G)$. If $D$ is a set of vertices
for which $\Gamma_D(G)=\{id\}$, then we say that $D$ is a {\it
fixing set} of $G$. The \emph{fixing number} of a graph $G$ is
defined as the minimum cardinality of a fixing set, denoted by
$fix(G)$.

Throughout the paper, $\mathbb{V}$ denotes a vector space of
dimension $n$ over the field of $q$ elements and
$\{b_1,b_2,...,b_n\}$ be a basis of $\mathbb{V}$. The \emph{non-zero
component graph} of $\mathbb{V}$ \cite{Das}, denoted by $G(\mathbb{V})$, is a
graph whose vertex set consists of the non-zero vectors of
$\mathbb{V}$ and two vertices are joined by an edge if they share at
least one $b_{i}$ with non-zero coefficient in their unique linear
combination with respect to $\{b_{1}, b_{2},\ldots,b_{n}\}$. It is proved in \cite{Das} that $G(\mathbb{V})$ is
independent of the choice of basis, i.e., isomorphic non-zero
component graphs are obtained for two different bases. In
\cite{Das}, Das studied automorphisms of $G(\mathbb{V})$. It is
shown that an automorphism maps basis of $G(\mathbb{V})$ to a basis
of a special type, namely non-zero scalar multiples of a permutation
of basis vectors.
\begin{Theorem}\cite{Das}
Let $\varphi:G(\mathbb{V})\rightarrow G(\mathbb{V})$ be a graph
automorphism. Then, $\varphi$ maps a basis $\{\alpha_1,
\alpha_2,...,\alpha_n\}$ of $\mathbb{V}$ to another basis $\{\beta_1,
\beta_2,...,\beta_n\}$ such that there exists a permutation $\sigma$
from the symmetric group on $n$ elements, where each $\beta_i$ is of the
form $c_i\alpha_{\sigma(i)}$ and each $c_i's$ are non-zero.
\end{Theorem}

The \emph{skeleton} of a vertex $u\in V(G(\mathbb{V}))$ denoted by
$S_u$, is the set of all those basis vectors of $\mathbb{V}$ which
have non-zero coefficients in the representation of $u$ as the
linear combination of basis vectors. In \cite{Murtaza}, we partition
the vertex set of $G(\mathbb{V})$ into $n$ classes $T_i$, $(1\le i
\le n)$, where $T_i=\{v\in \mathbb{V}: |S_v|=i\}$. For example, if
$n=4$ and $q=2$, then $T_3=\{b_1+b_2+b_3,b_1+b_2+b_4,b_1+b_3+b_4,b_2+b_3+b_4\}$.

The section wise break up of the article is as follows: In Section
2, we study the properties of automorphisms of non-zero component
graph. We discuss the relation between vertices and their images in
terms of their skeletons. In Section 3, we discuss some
properties and the cardinality of fixing neighborhood of pair of
vertices of non-zero component graph. The last section is devoted to
the study of the fixed number of graphs. We find the fixed number of
non-zero component graph. We give a realizable result about the
existence of a graph $G$, for every positive integer $N$, such that
$fxd(G)-fix(G)\ge N$, where $fxd(G)$ is the fixed number of graph.
\section{Automorphisms of non-zero component graph}
In this section, $\mathbb{V}$ is a vector space of dimension $n\geq
3$ over the field of $2$ element and $G(\mathbb{V})$ is the
corresponding non-zero component graph.

%\begin{Lemma}\label{CarNbr}
%Let $\mathbb{V}$ be a vector space of dimension $n$ over a field
%$\mathbb{F}$ of $2$ elements. If $v\in T_s$ for $s$ $(1\le s \le
%n)$, then for an $r$ $(1\le r \le n)$
%$$|N(v)\cap T_r|=\left\{
%\begin{array}{ll}
%{n\choose r}-{{n-s}\choose r}-1\hskip 0.8cm if \hskip 1cm r\le n-s\,\,\, \mbox{and}\,\,\,r=s  \\
%{n\choose r}-{{n-s}\choose r}\hskip 1.4cm if \hskip 1cm r\le n-s\,\,\, \mbox{and}\,\,\,r\ne s \\
%{n\choose r}-1\hskip 2.1cm if \hskip 1cm n-s<r\le n\,\,\, \mbox{and}\,\,\,r=s \\
%{n\choose r}\hskip 2.7cm if \hskip 1cm n-s<r\le n\,\,\, \mbox{and}\,\,\,r\ne s.\\
%\end{array}
%\right.$$
%\end{Lemma}

\begin{Lemma}\label{Murtaza}\cite{Murtaza}
If $v\in T_s$ for $s$ $(1\le s \le n)$, then
$deg(v)=(2^s-1)2^{n-s}-1$.
\end{Lemma}

\begin{Lemma}\label{VertexDeg}
Let $u,v\in V(G(\mathbb{V}))$ such that $u\in T_r$ and $v\in T_s$
where $r\ne s$ and $1\le r,s \le n$, then $u\not\sim^g v$ for all
$g\in \Gamma(G(\mathbb{V}))$.
\end{Lemma}

\begin{proof}
Since $r\ne s$ $\Rightarrow$ $(2^r-1)2^{n-r}-1\ne (2^s-1)2^{n-s}-1$.
Thus $deg(u)\ne deg(v)$ by Lemma \ref{Murtaza} and hence
$u\not\sim^g v$.
\end{proof}
From Lemma \ref{VertexDeg}, we have the following straightforward
remarks.
\begin{Remark}
Since $\sum\limits_{i=1}^n b_i \in T_n$ is the only element in
$T_n$. Therefore, by Lemma \ref{VertexDeg}, $g(\sum\limits_{i=1}^n
b_i)=\sum\limits_{i=1}^n b_i$ for all $g\in \Gamma(G(\mathbb{V}))$.
\end{Remark}
\begin{Remark}
Let $u,v\in V(G(\mathbb{V}))$, then $(u,v)\in S(G(\mathbb{V}))\times
S(G(\mathbb{V})) $ if and only if both $u,v\in T_i$ for some $i$, $1\leq
i\leq n-1$.
\end{Remark}
\begin{Lemma}
Let $b_l\in T_1$ be a basis vector and $g\in \Gamma_{b_l}$. Let
$u\in V(G(\mathbb{V}))$, then $b_l\in S_u$ if and only if $b_l\in
S_{g(u)}$.
\end{Lemma}

\begin{proof}\label{basisfixed}
Let $b_l\in S_u$, then $u$ is adjacent to $b_l$. Suppose on contrary
$b_l\not\in S_{g(u)}$, then $g(u)$ is not adjacent to $b_l=g(b_l)$,
a contradiction. Conversely, let $b_l\in S_{g(u)}$, then $b_l$ is
adjacent to $g(u)$. Suppose on contrary $b_l\not \in S_u$, then $u$
is not adjacent to $b_l=g(b_l)$, a contradiction.
\end{proof}

\begin{Lemma}\label{LemTransOnlyNghbr}
Let $b_l,b_m\in T_1$ be two distinct basis vectors of $\mathbb{V}$
and $g\in \Gamma(G(\mathbb{V}))$ be an automorphism such that
$b_l\sim^g b_m$. Let $u\in V(G(\mathbb{V}))$, then we have:
\begin{enumerate}[label=(\roman*)]
  \item If $b_l\in S_u$ and $b_m\not\in S_u$, then
$b_l\not\in S_{g(u)}$ and $b_m\in S_{g(u)}$.
  \item $b_l,b_m\in S_u$ if and only if
$b_l,b_m \in S_{g(u)}$.
\end{enumerate}
\end{Lemma}

\begin{proof}
$(i)$ If $b_l\in S_u$ and $b_m\not\in S_u$, then $u$ is adjacent to
$b_l$ and non-adjacent to $b_m$. We discuss four possible cases:
\begin{enumerate}
        \item If both $b_l,b_m\not\in S_{g(u)}$, then
        $g(u)$ is non-adjacent $b_m=g(b_l)$, a contradiction.
        \item If both $b_l,b_m\in S_{g(u)}$, then
        $g(u)$ is adjacent to $b_l=g(b_m)$, a contradiction.
        \item If $b_l\in S_{g(u)}$ and $b_m\not\in S_{g(u)}$, then
         $g(u)$ is adjacent to $b_l=g(b_m)$, a contradiction.
        \item If $b_l\not\in S_{g(u)}$ and $b_m\in S_{g(u)}$, then $g(u)$
         is adjacent to $b_m=g(b_l)$ and non-adjacent
        to $b_l=g(b_m)$.
\end{enumerate}
As in Case (4), $g$ is preserving the relation of adjacency
and non-adjacency, hence proved.\\
$(ii)$ Let $b_l,b_m\in S_u$, then $u$ is adjacent to both $b_l$ and
$b_m$. We discuss four possible cases:
    \begin{enumerate}
        \item If both $b_l,b_m\not\in S_{g(u)}$, then
        $g(u)$ is non-adjacent $b_m=g(b_l)$, a contradiction.
        \item If both $b_l,b_m\in S_{g(u)}$, then
        $g(u)$ is adjacent to $b_l=g(b_m)$ and $b_m=g(b_l)$.
        \item If $b_l\in S_{g(u)}$ and $b_m\not\in S_{g(u)}$, then
         $g(u)$ is non-adjacent to $b_m=g(b_l)$, a contradiction.
        \item If $b_l\not\in S_{g(u)}$ and $b_m\in S_{g(u)}$, then $g(u)$
         is non-adjacent to $b_l=g(b_m)$, a contradiction.
    \end{enumerate}
 Since, $g$ preserves the relation of adjacency in Case (2), therefore $b_l,b_m\in S_{g(u)}$.
Converse part can be proved by the similar arguments.
\end{proof}

\begin{Lemma}\label{skelotondiff}
Let $u,v\in T_i$ for some $i$ $(1\le i \le n-1)$ and $u\sim^g v$ for
some $g\in \Gamma(G(\mathbb{V}))$. The following statements hold:
\begin{enumerate}[label=(\roman*)]
  \item If $b\in S_u \cap S_v$, then $g(b)\in S_u \cap S_v$.
  \item If $b\in S_u-S_v$, then $g(b)\in S_v-S_u$.
\end{enumerate}
\end{Lemma}

\begin{proof}
$(i)$ Since $b\in S_u \cap S_v$, therefore $b$ is adjacent to both
$u$ and $v$. Suppose on contrary $g(u)\not \in S_u \cap S_v$, then
$g(u)$ is non-adjacent to either $u=g(v)$ or $v=g(u)$, a
contradiction.\\
$(ii)$ Let $b\in S_u-S_v$, then $b$ is adjacent to $u$ and non-adjacent
 to $v$. Since $g$ is an automorphism, therefore $g(b)$ must
be adjacent to $g(u)=v$ and non-adjacent to $g(v)=u$. Hence,
$g(b)\in S_v-S_u$.
\end{proof}

\begin{Lemma}\label{stabskel}
Let $u\in V(G(\mathbb{V}))$ and $b\in S_u$ (or $b\notin S_u$). If
$g\in \Gamma_u$, then $g(b)\in S_u$ (or $g(b)\notin S_u$).
\end{Lemma}
\begin{proof}
Proof follows from the fact that $b$ and $u$ are
adjacent(non-adjacent), therefore $g(b)$ must be
adjacent(non-adjacent) to $g(u)=u$.
\end{proof}
%\begin{Lemma}
%Let $u,v\in T_i$ for some $i$ $(1\le i \le n-1)$. The number of
%automorphisms that maps $u$ and $v$ on each each other is $i -
%|S_u\cap S_v|$, where $S_u$ and $S_v$ are the skeletons of $u$ and
%$v$, respectively.
%%Let $A=\{g\in \Gamma(G(\mathbb{V})): u\sim^g v\}$,
%%then $|A|=i - |S_u\cap S_v|$.
%\end{Lemma}

%Since an automorphism of a graph preserves the relation of adjacency
%of vertices, therefore if $u\sim^g v$, then $g$ maps common
%neighbors and common non-neighbors of $u$ and $v$ on either common
%neighbors of $u$ and $v$ or common non-neighbors of $u$ and $v$.
%Similarly, the only neighbors of $u$ but not neighbors of $v$ are
%mapped on the only neighbors of $v$ but not neighbors of $u$ and
%vice versa. We use these arguments to prove the following lemma.
\begin{Lemma}
Let $\mathbb{V}$ be a vector space of dimension $n\ge 4$. Let
$b_l,b_m\in T_1$ be any two basis vectors where $m\ne l$ and $1\le
m,l\le n$. Let $b_l \sim^g b_m$ for some $g\in
\Gamma(G(\mathbb{V}))$.
\begin{enumerate}[label=(\roman*)]
\item Let $u\in T_2$ be such that $S_u=\{b_l, b_m\}$, then $g\in
\Gamma_u$.
\item Let $u\in T_{n-2}$ be such that $S_u=T_1-\{b_l,b_m\}$, then
$g\in \Gamma_u$.
\item Let $u,v\in T_{n-1}$ be such that $S_u=T_1-\{b_l\}$ and
$S_v=T_1-\{b_m\}$, then $u\sim^g v$.
\end{enumerate}
\end{Lemma}
\begin{proof}
$(i)$ As $b_l,b_m\in S_u$, therefore by Lemma
\ref{LemTransOnlyNghbr}(ii), $b_l,b_m\in S_{g(u)}$. Since, $u\in
T_2$ is the only element of $T_2$  which have both $b_l$ and $b_m$
in its
skeleton, therefor $g(u)=u$.\\
$(ii)$ As $b_l,b_m \not\in S_u$, therefore by Lemma
\ref{LemTransOnlyNghbr}(ii), $b_l,b_m\not\in S_{g(u)}$. Since $u$ is
the only element of $T_{n-2}$, which does not have both $b_l$ and
$b_m$ in its skeleton. Therefore, $g(u)=u$.\\
$(iii)$ As $b_m\in S_u$ and $b_l\not\in S_u$, therefore by Lemma
\ref{LemTransOnlyNghbr}(i), $b_m\not\in S_{g(u)}$ and $b_l\in
S_g(u)$. Since $v$ is the only element of $T_{n-1}$, which have
$b_l$ and does not have $b_m$ in its skeleton, therefore $u\sim^g
v$.
\end{proof}
\section{The Fixing neighborhood of non-zero component graph}
A vertex $x\in S(G)$ is said to {\it fix a pair} $(u,v)\in
S(G)\times S(G)$, if $\mathcal{O}_x(u)\ne \mathcal{O}_x(v)$ in $G$.
For $(u,v)\in S(G)\times S(G)$, the set $fix(u,v)=\{x\in S(G):
\mathcal{O}_x(u)\ne \mathcal{O}_x(v)\}$ is called the {\it fixing
neighborhood} of pair $(u,v)$. For any two distinct vertices $u$ and
$v$ in $G$ with $\mathcal{O}(u)\neq \mathcal{O}(v)$,
$fix(u,v)=\emptyset$.

\begin{Lemma}\label{LemBasFixNbhd}
Let $b_l,b_m\in T_1$ be two distinct basis vectors of a vector
 space $\mathbb{V}$ of dimension $n\ge 3$ over the field of 2 elements and $u\in V(G(\mathbb{V}))$ be such that either $b_l\in
S_u$ or $b_m\in S_u$. If $g\in \Gamma_u(G(\mathbb{V}))$, then
$b_l\not\sim^g b_m$. Moreover, $fix(b_l,b_m)=\{u\in
S(G(\mathbb{V})):$ either $b_l\in S_u$ or $b_m\in S_u \}.$
\end{Lemma}

\begin{proof}
Without loss of generality, assume $b_l\in S_u$ and $b_m\not\in
S_u$, then $u=g(u)$ is adjacent to $b_l$ and non-adjacent $b_m$.
Suppose on contrary $b_l\sim^g b_m$, then $b_m=g(b_l)$ must be
adjacent to $u=g(u)$, a contradiction. Hence $b_l\not\sim^g b_m$ and
consequently, $fix(b_l,b_m)=\{u\in S(G(\mathbb{V})):$ either $b_l\in
S_u$ or $b_m\in S_u \}.$
\end{proof}
\begin{Theorem}\label{fixNghTh}
Let $\mathbb{V}$ be a vector space of dimension $n\ge 3$ over the field
of 2 elements and $G(\mathbb{V})$ be its non-zero component graph.
Let $u,v\in T_i$ for some $i$ $(1\leq i\leq n-1)$, then
$fix(u,v)=\{w\in S(G(\mathbb{V})):|S_w\cap S_u|\neq |S_w\cap
S_v|\}$.
\end{Theorem}
\begin{proof}
Let $w\in S(G(\mathbb{V}))$ be such that $|S_w\cap S_u|\neq |S_w\cap
S_v|$, then
there are two possible cases:\\
Case 1: Either $S_w\cap S_u=\emptyset$ or $S_w\cap S_v=\emptyset$. Without
loss of generality, consider $S_w\cap S_u=\emptyset$. Let $g\in
\Gamma_w$. We claim that $\mathcal{O}_w(u)\ne \mathcal{O}_w(v)$
i.e., $u\not\sim^gv$. Suppose on contrary $u\sim^gv$. Since $S_w\cap
S_v\neq\emptyset$, therefore let $b\in S_w\cap S_v$. Then $b\in S_v-S_u$.
By Lemma \ref{skelotondiff}$(ii)$, $g(b)\in S_u-S_v$. Also $g(b)\in
S_w$ by Lemma \ref{stabskel}. Hence $g(b)\in
S_u\cap S_w$, a contradiction that $S_w\cap S_u=\emptyset$.\\
Case 2: Both $S_w\cap S_u\neq \emptyset$ and $S_w\cap S_v\neq \emptyset$.
Since $|S_u|=|S_v|$ and $|S_w\cap S_u|\neq |S_w\cap S_v|$, therefore
$|S_u\cup S_w|\neq |S_v\cup S_w|$ and hence $|S_v-\{S_u\cup
S_w\}|\neq |S_u-\{S_v\cup S_w\}|$. Moreover, $S_v-\{S_u\cup S_w\}$
and $S_u-\{S_v\cup S_w\}$ are disjoint sets of basis vectors and at
least one of them is non-empty. Without loss of generality assume
$S_v-\{S_u\cup S_w\}$ is non-empty and let $b\in S_v-\{S_u\cup
S_w\}$. For $g\in \Gamma_w$, we claim that $\mathcal{O}_w(u)\ne
\mathcal{O}_w(v)$ i.e., $u\not\sim^gv$. Suppose on contrary
$u\sim^gv $. Since $b\in S_v$ and $b\notin S_u\cup S_w$, therefore
$b\in S_v-S_u$ and $b\notin S_w$. By Lemma \ref{skelotondiff}$(ii)$,
$g(b)\in S_u-S_v$ and by Lemma \ref{stabskel} $g(b)\notin S_w$. Thus
$g(b)\in S_u-\{S_v\cup S_w\}$. Thus, if $b\in S_v-\{S_u\cup S_w\}$,
then $g(b)\in S_u-\{S_v\cup S_w\}$, which is a contradiction as
$S_v-\{S_u\cup S_w\}$ and $S_u-\{S_v\cup S_w\}$ are disjoint sets
with $|S_v-\{S_u\cup S_w\}|\neq |S_u-\{S_v\cup S_w\}|$ and $g$ is a
bijective function.
\end{proof}
\begin{Lemma}\label{FNP}
Let $G(\mathbb{V})$ be the non-zero component graph of a vector space
$\mathbb{V}$ of dimension $n\ge 3$ over the field of 2 elements. Let
$u,v\in T_i$ for some $i$ $(1\leq i\leq n-1)$, then we have:
\begin{enumerate}[label=(\roman*)]
  \item $fix(u,v)\cap \{S_u\cap S_v \}=\emptyset$
  \item $fix(u,v)\cap T_1= \{S_u\cup S_v\}-\{S_u \cap S_v\}$.
\end{enumerate}
\end{Lemma}
\begin{proof}
$(i)$ Let $b\in \{S_u\cap S_v\}$, then $S_b=\{b\}\subset \{S_u\cap
S_v\}$. Since $|S_u|=|S_v|$, therefore $|S_b\cap S_u|=|S_b\cap
S_v|$. Hence by Theorem \ref{fixNghTh}, $b\notin fix(u,v)$. Thus
$\{S_u\cap S_v\} \not\subset
fix(u,v) $.\\
Conversely, let $b\in fix(u,v)$ is a basis vector. Since
$|S_u|=|S_v|$ and by Theorem \ref{fixNghTh}, $|S_b\cap S_u|\neq
|S_b\cap S_v|$ implies that $\{b\}=S_b\not\subset \{S_u\cap S_v\}$,
thus $b\notin \{S_u\cap S_v\} $. Therefore, $fix(u,v)\not\subset
\{S_u\cap S_v\}$. Hence
$fix(u,v)\cap \{S_u\cap S_v \}=\emptyset$.\\
$(ii)$ Let $b\in fix(u,v)\cap T_1$, then by Theorem \ref{fixNghTh},
$|S_b\cap S_u|\neq |S_b\cap S_v|$, also $|S_u|=|S_v|$, which implies
that at least one of $S_b\cap S_u$ and $S_b\cap S_v$ is non-empty.
Also by $(i)$, $b\notin S_u\cap S_v$. Thus either $b\in S_u-S_v$ or
$b\in S_v-S_u$. Hence, $b\in \{S_u\cup S_v\}-\{S_u \cap S_v\} $.\\
Conversely, let $b\in \{S_u\cup S_v\}-\{S_u \cap S_v\} $. Then
either $S_b\subset S_u$ or $S_b\subset S_v$. In both cases $|S_b\cap
S_u|\neq |S_b\cap S_v|$, as $|S_u|=|S_v|$. Hence by Theorem
\ref{fixNghTh}, $b\in fix(u,v)$. Since $b$ is basis vector,
therefore $b\in fix(u,v)\cap T_1$.
\end{proof}

\begin{Theorem}\label{fixnbhtranslation}
Let $G(\mathbb{V})$ be the non-zero component graph of a vector space
$\mathbb{V}$ of dimension $n\ge 3$ over the field of 2 elements. Let
$u,v\in T_i$ for some $i$ $(1\leq i\leq n-1)$, then
$$fix(u,v)=fix(u-\sum\limits_{b\in S_u\cap S_v}b, v-\sum\limits_{b\in S_u\cap S_v}b).$$
\end{Theorem}
\begin{proof}
\begin{flushleft}
Let $w\in fix(u,v)$, then by Theorem \ref{fixNghTh}, $|S_u\cap
S_w|\neq |S_v\cap S_w| \Leftrightarrow S_u\cap S_w \neq S_v\cap S_w$
$\Leftrightarrow \{S_u\cap S_w\} -\{S_u\cap S_v\} \neq \{S_v\cap
S_w\}-\{S_u\cap S_v\}$
$\Leftrightarrow \{S_u-(S_u\cap S_v)\}\cap \{S_w -(S_u\cap S_v)\} \neq \{S_v-(S_u\cap S_v)\}\cap \{S_w -(S_u\cap S_v)\}$.\\
Since $fix(u,v)\cap \{S_u\cap S_v \}=\emptyset$ by Lemma \ref{FNP}$(i)$. Therefore,\\
$\Leftrightarrow \{S_u-(S_u\cap S_v)\}\cap S_w \neq \{S_v-(S_u\cap
S_v)\}\cap S_w$\\
$\Leftrightarrow |\{S_u-(S_u\cap S_v)\}\cap S_w| \neq
|\{S_v-(S_u\cap S_v)\}\cap S_w|$. Hence by Theorem \ref{fixNghTh},
$\Leftrightarrow w\in fix(u-\sum\limits_{b\in S_u\cap S_v}b,
v-\sum\limits_{b\in S_u\cap S_v}b).$

\end{flushleft}

\end{proof}
\begin{Theorem}\label{fixTi}
Let $G(\mathbb{V})$ be the non-zero component graph of a finite
dimensional vector space $\mathbb{V}$ of dimension $n\geq 3$ over
the field of 2 elements. Let $u,v\in T_{i'}$ for some $i'$, $(1\leq
i'\leq n-1)$, such that $S_u\cap S_v=\emptyset$. Then

    \[|fix(u,v)\cap T_i|={n\choose i}-\sum\limits_{\substack{0\, \leq \, j \, \leq \, i'\\0 \, \leq \, k \, \leq \, n-2i'\\i \, = \, 2j \, + \, k}}{i'\choose j}^2 {n-2i'\choose k}-{n-2i'\choose
    i},\]
where $j,k$ are two numbers such that $i=2j+k$.
\end{Theorem}
\begin{proof}
First we count the vertices in $T_i-fix(u,v)$. Let $w\in
T_i-fix(u,v)$ for some $i$, $(1\leq i\leq n-1)$, then by Theorem
\ref{fixNghTh}, $|S_w\cap S_u |=|S_w\cap S_v |$. Since
$|S_u|=|S_v|=i'$, therefore $|S_w\cap S_u |=|S_w\cap S_v |=j$, where
$(0\leq j\leq i')$. If $j=0$, then $S_w\cap S_u=S_w\cap S_v=\emptyset$,
and $S_w$ contains $i$ basis vectors of $T_1-\{S_u\cup S_v\}$. Since
$i$ elements out of $n-2i'$ elements of $T_1-\{S_u\cup S_v \}$ can
be chosen in ${n-2i'\choose i}$ ways, therefore there are
$n-2i'\choose i$ elements $w$ in each $T_i$, $(1\leq i\leq n-2i')$
such that $S_w\cap S_u=S_w\cap S_v=\emptyset$ and hence these elements
belong to $T_i-fix(u,v)$. If $1\leq j\leq i'$, then the set $S_w$ of
cardinality $i$ contains $j$ vertices from each of two disjoint sets
$S_u$ and $S_v$ of cardinality $i'$ if and only if $i=2j+k$, where
$k=|S_w-\{S_u\cup S_v \}|$. As $j$ vertices can be chosen out of
$i'$ vertices of $S_u$ in $i'\choose j$  ways, $j$ vertices can be
chosen out $i'$ vertices of $S_v$ in ${i'\choose j}$ ways and the
remaining $k$ vertices of $S_w$ can be chosen out of $n-2i'$
vertices of $T_1-\{S_u\cup S_v \}$  in ${n-2i'\choose k}$ ways.
Thus, by the fundamental principle of counting, there are
${i'\choose j}^2 {n-2i'\choose k}$ vertices $w$ in each $T_i$ for
each possibility of the numbers $j$ and $k$, such that $|S_w\cap S_u
|=|S_w\cap S_v |=j$. Therefore, the number of vertices in each $T_i$
such that $|S_w\cap S_u |=|S_w\cap S_v |\neq 0$ is
$\sum\limits_{\substack{0\,\leq\, j\,\leq\, i'\\0\,\leq\, k\,\leq\,
n-2i'\\i\,=\,2j\,+\,k}}{i'\choose j}^2 {n-2i'\choose k}$. Hence, the proof
follows by the fact that the cardinality of $T_i$ for each $i$,
$(1\leq i\leq n-1)$ is ${n\choose i}$.
\end{proof}
\begin{Corollary}
Let $G(\mathbb{V})$ be the non-zero component graph of a finite
dimensional vector space $\mathbb{V}$ of dimension $n\geq 3$ over
the field of 2 elements. Let $u,v\in T_{i'}$ for some $i'$, $(2\leq
i'\leq n-1)$, such that $S_u\cap S_v\neq \emptyset$. If $r=i'-|S_u\cap
S_v|$, then

    $$|fix(u,v)\cap T_i|={n\choose i}-\sum\limits_{\substack{0\,\leq\, j\,\leq\, r\\0\,\leq\, k\,\leq\, n-2r\\i\,=\,2j\,+\,k}}{r\choose j}^2 {n-2r\choose k}-{n-2r\choose
    i},$$
where $j,k$ are two numbers such that $i=2j+k$.
\end{Corollary}

\begin{proof}
Since the skeletons of $u-\sum_{b\in S_u\cap S_v}b$ and
$v-\sum_{b\in S_u\cap S_v}b$ are disjoint and
$u-\sum_{b\in S_u\cap S_v}b,v-\sum_{b\in S_u\cap
S_v}b\in T_r$. Therefore by Theorem \ref{fixTi},
$|fix(u-\sum_{b\in S_u\cap S_v}b,v-\sum_{b\in S_u\cap
S_v}b)\cap T_i|={n\choose i}-\sum\limits_{\substack{0\,\leq \,j\,\leq\,
r\\0\,\leq k\,\leq\, n-2r\\i\,=\,2j\,+\,k}}{r\choose j}^2 {n-2r\choose
k}-{n-2r\choose i}$. Also by Theorem \ref{fixnbhtranslation}, $|fix(u,v)\cap
    T_i|=|fix(u-\sum_{b\in S_u\cap S_v}b,v-\sum_{b\in S_u\cap
S_v}b)\cap T_i|$. Hence proved.

\end{proof}

\section{The fixed number of graph}
In this section, we give a realizable result about the existence of a
graph in terms of the difference of fixed number and fixing number of
the graph. We find an upper bound on the size of fixing graph of a graph $G$.
We also find the fixed number of non-zero component graph.

The \emph{open neighbourhood} of a vertex $u$ in a graph $G$ is $N_G(u)=\{v\in V(G): v$ is adjacent to $u$ in $G\}$. The \emph{fixed number} of a graph $G$, $fxd(G)$, is the minimum number $k$ such that every subset of vertices of $G$ of cardinality $k$ is a fixing set of $G$. Note that $0\leq fix(G) \leq fxd(G) \leq |V(G)|-1$.

In the next theorem, we will prove the existence of a graph $G$ for
a given positive integer $N$ such that $fxd(G)-fix(G)\geq N$.
\begin{Theorem}
For every positive integer $N$, there exists a graph $G$ such that
$fxd(G)-fix(G)\geq N$.
\end{Theorem}
\begin{proof}
We choose $k\geq max\{3,\frac{N+3}{2}\}$. Let $G$ be a
bipartite graph and $V(G)=U\cup W$, where $U=\{u_1,...,u_{2^k-2}\}$ and ordered set
$W=\{w_1,w_2,...,w_{k-1}\}$ and both $U$ and $W$ are disjoint.
Before defining adjacency, we assign coordinates to each vertex of
$U$ by expressing each integer $j$ $(1\leq j \leq 2^k-2)$ in its
base 2 (binary) representation. We assign each $u_j$ $(1\leq j \leq
2^k-2)$ the coordinates $(a_{k-1},a_{k-2},...,a_0)$ where $a_m$
$(0\leq m \leq k-1)$ is the value in the $2^m$ position of binary
representation of $j$. For integers $i$ $(1\leq i \leq k-1)$ and $j$
$(1\leq j \leq 2^k-2)$, we join $w_i$ and
$u_j(a_{k-1},a_{k-2},...,a_0)$ if and only if $i=\sum_{m=0}^{k-1}
a_m$. This completes the construction of graph $G$.\par Next we will
prove that $fix(G)=2^k-(k+1)$. We denote $N(w_i)=\{u_j\in U: u_j$ is
adjacent to $w_i, (1\leq j \leq 2^{k-2})\}$ and it is obvious to see
that $N(w_i)\cap N(w_j)=\emptyset$ as if
$u(a_{k-1},a_{k-2},...,a_0)\in N(w_i)\cap N(w_j)$, then
$i=\sum_{m=0}^{k-1} a_m=j$. Number of vertices in each $N(w_i)$ is
the permutation of $k$ digits in which digit 1 is appears $i$ times
and digit 0 appears $(k-i)$ times, hence $|N(w_i)|={k\choose i}$. As
$N(w_i)\cap N(w_j)=\emptyset$, so a fixing set $D$ with the minimum cardinality
must have
${k\choose i}-1$ vertices from each $N(w_i)$ $(1\leq i\leq k-1)$,
for otherwise if $u,v\in N(w_i)$ and $u,v\not\in D$ for some $i$,
then there exists an automorphism $g\in \Gamma(G)$ such that
$g(u)=v$ because $u$ and $v$ have only one common neighbor $w_i$,
which is a contradiction that $D$ is a fixing set. Moreover
$D\subseteq U$ as each $w_i$ $(0\leq i \leq k-1)$, is fixed while
fixing at least ${k\choose i}-1$ vertices in each $N(w_i)$. Hence,
\[fix(G)=\sum_{i=1}^{k-1} {k \choose i}-(k-1)=2^k-(k+1)\]\par
Next we will find $fxd(G)$. As order of $G$ is $2^k+k-3$ and set of
all vertices of $G$ except one vertex forms a fixing set of $G$. It
can be seen that $fxd(G)=2^k+k-4$, for otherwise if $fxd(G)<2^k+k-4$
and $u,v \in N(w_i)$ for some $i$, then the set $E=W\cup U\setminus
\{u,v\}$ consisting of $2^k+k-5$ is not a fixing set, which implies
that $fxd(G)=2^k+k-4$. Hence for the graph $G$, we have
$fxd(G)-fix(G)=2k-3\geq N$ as required.
\end{proof}

Javaid \emph{et al.} \cite{ijav} defined the fixing graph of a graph $G$. The authors gave a method for finding fixing number of a graph with the help of its fixing
graph. The authors found a lower bound on the size of fixing graph of a graph for which $fxd(G)=fix(G)$. We give an upper bound on the size of fixing graph of such a graph. The {\it fixing graph}, $F(G)$, of a graph $G$ is a bipartite graph
with bipartition $(S(G),$ $V_s(G))$ and a vertex $x\in S(G)$ is
adjacent to a pair $(u,v)\in V_s(G)$ if $x\in fix(u,v)$. For a set
$D\subseteq S(G)$, $N_{F(G)}(D)=\{(u,v)\in V_s(G):$ $x\in fix(u,v)$
for some $x\in D\}$. In the fixing graph, $F(G)$, the minimum
cardinality of a subset $D$ of $V(G)$ such that $N_{F(G)}(D)=V_s(G)$
is the fixing number of $G$.
\begin{Theorem}\label{t1}
Let $G$ be a graph of order $n\geq 2$ such that $fix(G)=fxd(G)=k$, then
\center
$|E(F(G))|\leq n({n \choose 2}-k+1)$.\\
\end{Theorem}
\begin{proof}
Let $|S(G)|=r$ and $|V_s(G)|=s$, then $r\leq n$ and $s\leq {r\choose
2}\leq {n\choose 2}$. Let $v\in S(G)$. We will prove that
$deg_{F(G)}(v)\leq s-k+1$. Suppose $deg_{F(G)}(v)\geq s-k+2$, then
there are at most $k-2$ pairs in $V_s(G)$ which are not adjacent to
$v$. Let $V_s(G)\backslash
N_{F(G)}(v)=\{(u_{1},v_{1}),(u_{2},v_{2}),...,(u_{t},v_{t})\}$,
where $t\leq k-2$. Note that, $u_{i}$ fixes $(u_{i},v_{i})$ for each
$i$ with $1\leq i \leq t$. Hence, $u_{i}$ is adjacent to pair
$(u_{i},v_{i})$ in $F(G)$ for each $i$, $1\leq i \leq t$ and so
$N_{F(G)}(\{v,u_{1},u_{2},...,u_{t}\})=V_s(G)$. Hence, $fix(G) \leq
t+1 \leq k-1 $, which is a contradiction. Thus, $deg_{F(G)}(v) \leq
s-k+1\leq {n \choose 2}-k+1$, and consequently, $|E(F(G))| \leq n({n
\choose 2}-k+1)$.
\end{proof}
%\begin{Lemma}
%Let $F(G(\mathbb{V}))$ be fixing graph of a graph $G(\mathbb{V})$,
%where $\mathbb{V}$ is a vector space of dimension $n$ over the field
%of 2 elements. Let $b\in T_1$ be a basis vector, then
%$N_{F(G(\mathbb{V}))}=\{(u,v)\in S(G(\mathbb{V})):$ either $b\in
%S_u$ or $b\in S_v$ but not both$\}$.
%\end{Lemma}
%
%\begin{proof}
%Let $(u,v)\in S(G(\mathbb{V}))$ be such that either $b\in S_u$ or
%$b\in S_v$ but not both. Without loss of generality, consider $b\in
%S_u$ and $b\not\in S_v$. We claim that $u\not\sim ^g v$ for all
%$g\in \Gamma_b$. Suppose on contrary $u\sim^g v$ for some $g\in
%\Gamma_b$, $v=g(u)$ must be adjacent to $b=g(b)$, a contradiction.
%Thus, $u\not\sim^g v$ and hence $b\in fix (u,v)$.
%\end{proof}
%
%\begin{Theorem}
%Let $F(G(\mathbb{V}))$ be fixing graph of a graph $G(\mathbb{V})$,
%where $\mathbb{V}$ is a vector space of dimension $n$ over the field
%of 2 elements. Let $w\in T_i$ for some $i$, $1\le i \le n-1$, then
%then $N_{F(G(\mathbb{V}))}(T_i)=\{(u,v)\in V_s(G(\mathbb{V})):
%|S_w\cap S_u|\neq |S_w\cap S_v|\}$.
%\end{Theorem}
%\begin{proof}
%By definition of $N_{F(G(\mathbb{V}))}(T_i)$, for some $w\in T_i$,
%$(u,v)\in N_{F(G(\mathbb{V}))}(T_i)$ if $w\in fix(u,v)$. And by
%Theorem \ref{fixNghTh}, $w\in fix(u,v)$ if $|S_w\cap S_u|\neq
%|S_w\cap S_v|$. Hence, $(u,v)\in N_{F(G(\mathbb{V}))}(T_i)$ if
%$|S_w\cap S_u|\neq |S_w\cap S_v|$, for some $w\in T_i$.
%\end{proof}

A set of vertices $A\subset V(G)$ is referred as \emph{non-fixing} set if
$\Gamma_A(G)\setminus \{id\}\neq \emptyset$. The following remark is useful in
finding the fixed number of non-zero component graph of a graph $G$.
\begin{Remark}\label{fixdrem}\cite{ijav}
Let $G$ be a graph of order $n$. If $r$, $(1\leq r\leq n-2)$ is the
largest cardinality of a non-fixing subset of $G$, then
$fxd(G)=r+1$.
\end{Remark}

The following result gives the fixed number of non-zero component
graph.
\begin{Theorem}
Let $G(\mathbb{V})$ be the non-zero component graph of a finite dimensional vector space $\mathbb{V}$ of dimension $n\geq 3$ over the field of 2 elements,
then $fxd(G(\mathbb{V}))=2^{n-1}$.
\end{Theorem}
\begin{proof}
We choose a pair of basis vectors $b_l,b_m\in T_1$ and define a set
$A=\{x\in V(G(\mathbb{V})):$ either $b_l,b_m\in S_x$ or
$b_l,b_m\notin S_x\}$. Take $B=V(G(\mathbb{V}))-A=\{y\in V(G(\mathbb{V})):$ either $b_l\in S_y,b_m\notin S_y$ or $b_l\notin S_y,b_m\in S_y\}$. We choose two vertices $u,v\in B \cap T_i$
for some $i$ $(1\leq i\leq n)$, such that $S_u-S_v=\{b_l\}$ and
$S_v-S_u=\{b_m\}$. This implies $\{S_u-S_v\}\cup
\{b_m\}=\{S_v-S_u\}\cup \{b_l\}$. We prove that
$A$ is a non-fixing set of $G(\mathbb{V})$ by proving that there exist a non-trivial $g\in \Gamma_A$ such that $u\sim^g v$. Now let $w\in A$ be an arbitrary vertex of set $A$. We discuss two cases of $w$.\\
\textbf{Case 1}: If $b_l,b_m\in S_w$, then $b_l\in S_w\cap S_u$ and
$b_m\in S_w\cap S_v$. Also $\{ S_w\cap S_u\}\cup \{b_m\}=\{ S_w\cap
S_v\}\cup \{b_l\}$ implies $|\{ S_w\cap S_u\}\cup \{b_m\}|=|\{
S_w\cap S_v\}\cup \{b_l\}|$. This implies $|S_w\cap S_u|=|S_w\cap
S_v|$. Therefore by Theorem \ref{fixNghTh}, $w\notin fix(u,v)$ and
$u\sim^gv$ for some $g\in \Gamma_w$.\\
\textbf{Case 2}: If $b_l,b_m\notin S_w$, then $S_w\cap S_u= S_w\cap
S_v$ implies $|S_w\cap S_u|=|S_w\cap S_v|$. Therefore by Theorem
\ref{fixNghTh}, $w\notin fix(u,v)$ and $u\sim^gv$ for some $g\in \Gamma_w$.\\
Since the choice of $w$ is arbitrary in $A$. Therefore,
$\Gamma_A-\{id\}\neq \emptyset$. Hence, $A$ is a non-fixing set of
$G(\mathbb{V})$.
\par Next we prove that any superset of set $A$ is a fixing set of $G(\mathbb{V})$. Let $z\in B$ be an arbitrary vertex and $A'=A\cup \{z\}$.
Without loss of generality assume $b_l\in S_z$ and $b_m\notin S_z$.
We choose $u,v\in B\setminus\{z\}$ such that $S_u-S_v=\{b_l\}$ and
$S_v-S_u=\{b_m\}$, then $b_l\in S_z\cap S_u$ and $b_m\notin S_z\cap
S_v$. This implies $|S_z\cap S_u|\neq |S_z\cap S_v|$. Therefore by
Theorem \ref{fixNghTh}, $z\in fix (u,v)$. In particular $z\in
fix(b_l,b_m)$, thus $b_l\nsim^gb_m$ where $g\in \Gamma_{A'}$. Hence
$\Gamma_{A'}$ is trivial and $A'$ is a fixing set of
$G(\mathbb{V})$. Thus $A$ is
non-fixing set of $G(\mathbb{V})$ with the largest cardinality.
\par Since $T_i$ for each $i$ $(2\leq i\leq n)$, contains $n-2\choose
{i-2}$ elements which have both $b_l$ and $b_m$ in their skeletons.
Also $T_i$ for each $i$ $(1\leq i\leq n-2)$, contains $n-2\choose i
$ elements which do not have both $b_l$ and $b_m$ in their
skeletons. Therefore, $|A|=\sum\limits_{i=1}^{n-2}{n-2\choose
{i}}+\sum\limits_{i=2}^n {n-2 \choose {i-2}}= 2^{n-1}-1+2^{n-2}=
2^{n-1}-1$. Hence by Remark \ref{fixdrem},
$fxd(G(\mathbb{V}))=2^{n-1}$.

\end{proof}

\begin{Theorem}\label{fxd}\cite{ijav}
Let $G$ be a connected graph of order $n$. Then $fxd(G)=n-1$ if and
only if $N(v)\setminus \{u\}=N(u)\setminus \{v\}$ for some $u,v\in
V(G)$.
\end{Theorem}

Let $\mathbb{V}$ be a vector space of dimension $n$ over the field
of $q\geq 3$ elements and $G(\mathbb{V})$ be its non-zero component
graph. Then $T_i$, for each $i$ $(1\leq i\leq n)$, has $n\choose i$
sets of vertices each of cardinality $(q-1)^i$ whose skeletons are
same. Hence their neighborhoods are also same. Thus each of these
$n\choose i$ sets of vertices form twin classes of vertices. Thus
from Theorem \ref{fxd}, we have the following result:

\begin{Theorem}
Let $G(\mathbb{V})$ be a non-zero component graph of a vector space
$\mathbb{V}$ of dimension $n$ over the field of $q\geq 3$ elements,
then $fxd(G(\mathbb{V}))=|G(\mathbb{V})|-1$.
\end{Theorem}

\begin{proof}
  Since $G(\mathbb{V})$ has twin vertices, therefor result follows by Theorem \ref{fxd}.
\end{proof}


\begin{thebibliography}{33}
\bibitem {Alb}
M. O. Albertson, K. Collins, Symmetry breaking in graphs,
\textit{Electron. J. Combin.,} 3 (1996), R 18.
\bibitem{us}
U. Ali, S. A. Bokhary, K. Wahid, G. Abbas, On resolvability of a
graph associated to a finite vector space, \textit{J. Algebra
Appl.}, 1950029.
\bibitem{AM} S. Arumugam, V. Mathew, The fractional metric dimension
of graphs, \textit{Disc. Math.}, 312(2012), 1584-1590.
\bibitem{bates1}
C. Bates, D. Bondy, S. Perkins and P. Rowley, Commuting involution
graphs for symmetric groups, \emph{J. Algebra.}, 266(2003), 133-153.
\bibitem{Beck}
I. Beck, Coloring of commutative rings, \textit{J. Algebra.},
116(1988), 208-226.
\bibitem {HB}
H. Benish, I. Irshad, M. Feng, I. Javaid, The Fractional fixing
number of graphs, https://arxiv.org/abs/1610.09232.

\bibitem{NBiggs}
N. Biggs,
\newblock Algebraic graph theory,
\newblock {\em Cambridge university press.},
(1993).


\bibitem{bates2}
D. Bondy, The connectivity of commuting graphs, \emph{J. Combin.
Theory Ser. A.}, 113(2006), 995-1007.
\bibitem {bou}
D. L. Boutin, Identifying graph automorphism using determining set,
\textit{Electron. J. Combin.,} 13(1), Research Paper 78(electronic),
2006.

\bibitem {bou1}
D. L. Boutin, The determining number of cartesian product,
\textit{J. Graph Theory.,} 61(2), 2009, 77-87.

\bibitem{cac}
J. C\'{a}ceres, D. Garijo, L. Puertas, C. Seara, On the determining
number and the metric dimension of graphs,  \textit{Electron. J.
Combin.,} 17, 2010.

\bibitem{Cameron}
P. J. Cameron and S. Ghosh, The power graph of a finite group,
\textit{Disc. Math.}, 311(2011), 1220-1222.

\bibitem{GCPZ}
G. Chartrand, C. Poisson, P. Zhang, Resolvability and the upper
dimension of graphs, \textit{Computers and Maths. with Appl.}, 39,
2000, 19-28.
\bibitem{1}
G. Chartrand and L. Lesniak, Graphs and Digraphs, $3$rd ed.,
\textit{Chapman and Hall, London.,} 1996.
\bibitem{Chakrabarty}
I. Chakrabarty, S. Ghosh and M. K. Sen, Undirected power graphs of
semi group, \textit{Semi group. Forum.}, 78(2009), 410-426.

\bibitem{Das}
A. Das, Non-Zero component graph of a finite dimensional vector
space, \textit{Communi. in Algebra.}, 44(2016), 3918-3926.
\bibitem{Das1}
A. Das, On non-zero component graph of vector spaces over finite
fields, \textit{J. Algebra Appl.}, 2016, 1750007.

\bibitem {erw}
D. Erwin, F. Harary, Destroying automorphism by fixing nodes,
\textit{Disc. Math.,} 306, 2006, 3244-3252.
\bibitem{fa}
M. Fazil, Determining sets and related parameters in graphs, PhD
Thesis.
%\bibitem{feh}
%M. Fehr, S. Gosselin, O. R. Oellermann, The metric dimension of
%Cayley digraphs, \textit{Disc. Math.,} 306, 2006, 31-41.
\bibitem{gar}
D. Garijo, A. Gonz\'{a}lez and A. M\'{a}rquez. The difference
between the metric dimension and the determining of a graph,
\emph{Appl. Math. and Computation.,} 249, 2014, 487-501.


\bibitem{gib}
C. R. Gibbons, J. D. Laison, Fixing Numbers of Graphs and Groups,
\textit{Electron. J. Combin.,} 16 Research Paper 39, 2009.
\bibitem{haray}
 F. Haray, Methods of destorying the symmetries of a
graph, \textit{Bull. Malaysian Math. Sci. Soc.,} 24(2), 2001.

\bibitem{Rad}
N. Jafari Rad and S. H. Jafari, Results on the intersection graphs
of subspaces of a vector space, http://arxiv.org/abs/1105.0803v1.

\bibitem{ijav}
I. Javaid, M. Murtaza, M. Asif, F. Iftikhar, On the fixed number of
graphs, \textit{Bull. Iran. Math. Soc.}, 2017, url="http://bims.iranjournals.ir/article\_1103.html".

\bibitem{jav}
I. Javaid, M. Salman, M. A. Chaudhary, The resolving share in
graphs, \textit{arXiv preprint arXiv:1408.0132.}
\bibitem{Moghaddamfar}
A. R. Moghaddamfar, S. Rahbariyan and W. J. Shi, Certain properties
of the power graph associated with finit group, \textit{J. Algebra
Appl.}, 13(2014), 450040.
\bibitem{Murtaza}
M. Murtaza, I. Javaid and M. Fazil, Locating-Dominating Sets and
Identifying Codes of a Graph Associated to a Finite Vector Space,
arXiv preprint arXiv:1701.08537, (2017).

\bibitem{Talebi}
Y. Talebi, M.S. Esmaeilifar and S. Azizpour, A kind of intersection
graph of vector space, \emph{J. Disc. Math. Sci. Crypt.,} 12,
6(2009), 681-689.

\bibitem{tom}
I. Tomescu, M. Imran, Metric dimension and $R$-Sets of connected
graphs, \textit{Graphs Combin.}, 27(4), 2011, 585-591.


\end{thebibliography}
\end{document}